\documentclass[11pt,a4paper]{article}
\usepackage{amssymb,  amsthm, color, amsfonts, amsmath}
\usepackage{graphics}
\usepackage{latexsym}
\usepackage{hyperref}
\title{\vspace{-2cm}
Quasineutral limit, dispersion and oscillations for Korteweg type fluids
}
\author{\textbf{Donatella Donatelli}\\
        {\small Department of Information Engineering, Computer Science and Mathematics}\\
       {\small University of L'Aquila}\\
       {\small 67100 L'Aquila, Italy}\\
        $\scriptstyle\mathtt{donatella.donatelli@univaq.it}$\\
 \and        
  \textbf{Marcati Pierangelo}\\
        {\small Department of Information Engineering, Computer Science and Mathematics}\\
       {\small University of L'Aquila}\\
       {\small and GSSI - Gran Sasso Science Institute}\\
       {\small 67100 L'Aquila, Italy}\\
        $\scriptstyle\mathtt{pierangelo.marcati@univaq.it}$\\}

         \date{}

\newcommand{\e}{\varepsilon}		       
\newcommand{\R}{\mathbb{R}}

\newcommand{\ut}{\tilde{u}}
\newcommand{\st}{\tilde{\sigma}}
\newcommand{\Vt}{\tilde{\Phi}}
\newcommand{\utt}{\tilde{\ut}}
\newcommand{\stt}{\tilde{\st}}
\newcommand{\Vtt}{\tilde{\Vt}}

\newcommand{\rt}{\tilde{\rho}}
\newcommand{\rtt}{\tilde{\rt}}
\newcommand{\se}{\sigma^{\varepsilon}}
\newcommand{\la}{\lambda}
\newcommand{\rl}{\rho^{\lambda}}
\newcommand{\sls}{\sigma^{\lambda}}
\newcommand{\vl}{\Phi^{\lambda}}
\newcommand{\ul}{u^{\lambda}}
\newcommand{\el}{E^{\lambda}}

\newcommand{\dive}{\mathop{\mathrm {div}}}

\newtheorem{theorem}{Theorem}[section]
\newtheorem{Main}{Main Theorem}[]

\newtheorem{lemma}[theorem]{Lemma}
\newtheorem{proposition}[theorem]{Proposition}

\theoremstyle{definition}
\newtheorem{definition}[theorem]{Definition}
\newtheorem{remark}[theorem]{Remark}

\begin{document}
\maketitle
\begin{abstract}
In the setting of general initial data and whole space we perform a rigorous   analysis of  the  quasineutral limit for a hydrodynamical model of a viscous plasma with capillarity tensor represented by the Navier Stokes  Korteweg Poisson system. We shall provide a detailed mathematical description  of the convergence process by analyzing the dispersion of the fast oscillating acoustic waves. However the standard acoustic wave analysis is not sufficient to control the high frequency oscillations in the electric field but it is necessary to estimates the dispersive properties induced by the capillarity terms. Therefore by using these additional estimates we will be able to control, via compensated compactness, the quadratic nonlinearity of the stiff electric force field. In conclusion, opposite to the zero capillarity case \cite{DM12} where persistent space localized time high frequency oscillations need to be taken into account,  we show that as $\la\to 0$, the density fluctuation $\rl-1$  converges strongly to zero and the fluids behaves according to an incompressible dynamics. 
\medbreak 
\medbreak
\end{abstract}
\newpage
\tableofcontents
\newpage
\section{Introduction and plan of the paper}
\subsection{Introduction}
In the last years hydrodynamical models have been widely used to  describe physical phenomena in plasma physics. In the particular case where the  viscous stress tensors are taken into consideration the most simple model is provided by the coupling of the compressible  Navier Stokes equations with  the Poisson equation. In this case, in dimensionless units, the coupling  can be expressed in terms of a constant  $\la$ which represents the scaled Debye length, a characteristic physical parameter for plasmas related to the phenomenon of the so called ``Debye shielding'', \cite{GR95}. Moreover, if one wants to take into consideration  the surface tension effects it is necessary to add to the momentum equation a capillarity tensor, namely one has to consider Korteweg type model of capillarity. This type of models were first introduced by Korteweg \cite{K1901}, see also \cite{M52} and derived rigorously by Dunn and Serrin \cite{DS85} and are based on an extended version of thermodynamics which assumes that the energy of the fluid not only depends on standard variables but also on the gradient of the density. Finally the model we will consider  in this paper is given by the following Navier-Stokes -Poisson Korteweg system in $3-D$, namely
\begin{equation}
\partial_{t}{\rho^{\lambda}}+\dive(\rho^{\lambda} u^{\lambda})=0, \label{1} 
\end{equation}
\begin{equation}
\partial_{t}(\rho^{\lambda} u^{\lambda})+\dive(\rho^{\lambda} u^{\lambda}\otimes u^{\lambda})+\nabla p(\rl)^{\gamma}=\dive(\mu\rho^{\lambda}D(u^{\lambda})+K(\rho^{\lambda}))+\rho^{\lambda}\nabla \Phi^{\lambda}, \label{2}
\end{equation}
\begin{equation}
 \lambda^2 \Delta \Phi^{\lambda} = \rho^{\lambda} - 1, \label{3}
\end{equation}
where $p(\rho^{\lambda})$ denotes the pressure term,
$$p(\rl)=(\rho^{\lambda})^{\gamma}, \qquad \gamma\geq 3/2,$$
 $K$  the capillarity tensor which is given by
\begin{equation}
K_{ij}(\rl)=\frac{\kappa}{2}(\Delta (\rl)^{2}-|\nabla \rl|^{2})\delta_{ij}-\kappa\partial_{i}\rl\partial_{j}\rl
\end{equation}
and $D(\ul)$ the strain tensor which has the form
\begin{equation}
D(\ul)_{ij}=\frac{\partial_{i}\ul+\partial_{j}\ul}{2}.
\end{equation}
Let $x\in\R^{3}$, $t\geq 0$, we denote by $\rl(x,t)$ the {\em negative charge density}, by $m(x,t)=\rl(x,t)\ul(x,t)$ the {\em current density},  by $\ul(x,t)$ the {\em velocity field}, by $\vl(x,t)$ the {\em electrostatic potential}, $\mu$ the {\em shear viscosity}.
The parameter $\lambda$ is the so called {\em Debye length} (up to a constant factor), $\kappa$ is the capillarity coefficient. Moreover let us observe that
 $$\dive K(\rho^{\lambda})=\kappa\rl\nabla\Delta \rl.$$
The purpose of this paper is to perform a rigorous limiting analysis when $\la\to 0$. The physical meaning of the Debye length $\lambda$ is  the distance over which the usual Coulomb field is killed off exponentially by the polarization of the plasma. In terms of physical variables the Debye length can be expressed as
\begin{equation}
\la=\la_{D}/L  \qquad \la_{D}=\sqrt{\frac{\e_{0}k_{B}T}{e^{2}n_{0}}},
\end{equation}
where $L$ is the macroscopic length scale, $\e_{0}$ is the vacuum permittivity, $k_{B}$ the Boltzmann constant, $T$ the average plasma temperature, $e$ the absolute electron charge and $n_{0}$ the average plasma density.  In many cases the Debye length is very small compared to the macroscopic length $\la_{D} << L$ and so it makes sense to consider the quasineutral limit $\la \to 0$ of the system \eqref{1}-\eqref{3}. In this situation the particle density is constrained to be close to the background density (equal to one in our case) of the oppositely charged particle. The limit $\la \to 0$ is called the quasineutral limit since the charge density almost vanishes identically. The velocity of the fluid then evolves according to an incompressible flow.

In the last years the  quasineutral limit  for hydrodynamical models of plasma or semiconductor devices has been widely studied by many authors, in the case of Euler Poisson system for instance by  \cite{CG00},  \cite{CDMS96}, \cite{L05},  \cite{PWY06} or  the case of the Navier Stokes Poisson system by \cite{W04}	who studied the quasineutral limit for the smooth solution with well-prepared initial data. Jiang and Wang   \cite{JW06} studied the combined quasineutral and inviscid limit of the compressible Navier- Stokes-Poisson system for weak solution and obtained the convergence of Navier- Stokes-Poisson system to the incompressible Euler equations with general initial data. Moreover in  \cite{JW06} the vanishing of viscosity coefficient was required in order to take the quasineutral limit and no convergence rate was derived there. The paper \cite{JLW08} studied the quasineutral limit of the isentropic Navier-Stokes-Poisson system both in the whole space and in the torus without restrictions on the viscous coefficients, with well prepared initial data. 

The authors in  \cite{DM08}  investigated the quasineutral limit of the isentropic Navier-Stokes-Poisson system in the whole space  and obtained the convergence of weak solution of the Navier-Stokes-Poisson system to the weak solution of the incompressible Navier- Stokes equations by means of dispersive estimates of Strichartz's type under the assumption that the Mach number is related to the Debye length. A more general analysis  in the context of weak solutions and in framework of general initial data was performed by the authors in \cite{DM12} where all  the regularity and smallness assumptions of the previous paper were removed. They were able to provide a detailed mathematical description of the convergence process by using microlocal defect measures and by developing an explicit correctors analysis. 

Finally,  in the contest of combined quasineutral and relaxation time limit we have the papers by Gasser and Marcati in \cite{GM01a, GM01b, GM03}. Other similar limits have been investigated in \cite{CDM13}, \cite{DFN10}, \cite{DoFeNo12}.

As far as it concerns the quasineutral limit for the compressible Navier Stokes Poisson Korteweg system  we refer to \cite{LY14}  for the $H^{s}$   setting of strong solutions  and we refer  to \cite{BDD05} for the quasineutral limit in a periodic domain,  where the electrons are assumed to be thermalized and to follow a non dimensional Maxwell-Boltzmann distribution $(\rl=e^{\vl})$.

In all of these papers but \cite{DM12}, the assumptions are designed to kill the presence of high frequency oscillations of the electric force fields. In this paper we are interested to understand the limiting behaviour in the same general situation of \cite{DM12}, when a Korteweg tensor is added.

In this paper we perform the zero Debye limit  for the Navier Stokes Poisson Korteweg system in the whole space and in the framework of weak solutions and large non smooth  initial data.  We do not make any particular assumption and our methods will control the quadratic stiff term of the electric fields by a better understanding of the role of the various scales in the oscillating wave packets.

A common feature for the limit analysis in the case of ill prepared initial data is the high plasma oscillations, namely the presence of high frequency time oscillations of the acoustic waves, moreover  what actually makes the limiting behavior analysis very hard is the presence of very stiff terms due to the electric field, whose oscillations cannot be controlled only by the dispersion of the acoustic waves, as pointed out in \cite{DM12}. In the case of fluids of Korteweg type  an additional difficulty is represented by the loss of information on the gradient of the velocity when vacuum appears and the presence of these phenomena causes the lost of compactness for the momentum term.
So it is particularly important to understand the different behaviors of the various vector fields acting in our system,  what and which are the relationship between high frequency interacting waves, dispersive behavior and the different roles of time and space oscillations. There are distinct dispersive behavior acting on distinct scales and one has to analyze in detail their behavior.

The classical acoustic wave analysis  is able to control  how the velocity field  disperses and oscillates and in detail it follows by analyzing the dispersion of the acoustic equation related to the plasma fluctuation. We get that  the dispersive behavior dominates on the high frequency time oscillations and  usual estimates of Strichartz  type are sufficient to pass into the limit of the convective term, but it is not able to control the electric fields time high frequency oscillations. 
The quadratic terms due to the electric force field  may not be analyzed in the same way since the dispersion may no longer dominates the time high frequency wave packets and we have to take care of the self-interacting waves.  The capillarity term induces  additional dispersive effects on a different scale than the usual acoustic waves and by using non standard Strichartz  estimates for the beam equations we can control the electric field. Intuitively while the acoustic waves scale like the standard D'Alembert equation, the capillarity tensor induces linear dispersive waves which scale like the Klein-Gordon or the Schr\"odinger equation.  The limit behavior of the quadratic nonlinearity of the electric force field is then deduced by a Compensated Compactness argument.
 The structure of this paper, as well as the main ingredients of our approach to this limiting process, can be summarized as follows. In Section 2 we set up our problem and state the main result. In Section 3 we collect the main mathematical tools needed in the paper, including notations and dispersive estimates.
 The Section 4 is devoted to obtain a priori estimates independent of $\la$, namely standard energy bounds. Section 5 concerns the convergence of the density. Section 6 and 7 are devoted to the convergence of the momentum and the electric field respectively. In that sections a careful analysis of the dispersion of the acoustic waves related to the plasma fluctuation  will be performed. Finally, in Section 8  we conclude with the proof of our Main Theorem \ref{tm1}.
 
\section{Statement of the problem and Main Result}

Before performing our limiting analysis, we recast our problem in a more precise way and we recall some  results concerning the existence of weak solutions for the Navier Stokes Poisson Korteweg  system. The system under consideration in this paper is given by the following equations, 
\begin{align}
\partial_{t}{\rho^{\lambda}}+\dive(\rho^{\lambda} u^{\lambda})&=0\notag\\
\partial_{t}(\rho^{\lambda} u^{\lambda})+\dive(\rho^{\lambda} u^{\lambda}\otimes u^{\lambda})+\nabla (\rho^{\lambda})^{\gamma}&=\dive(\mu\rho^{\lambda}D(u^{\lambda}))+\rho^{\lambda}\nabla \Phi^{\lambda}+\kappa\rho^{\lambda}\nabla\Delta\rho^{\lambda}\notag\\
\lambda^{2}\Delta \Phi^{\lambda}&=\rho^{\lambda}-1.
\label{3.1}
\end{align}
From now on we set $\mu=\kappa=$1 and we denote by $\pi^{\la}$ the renormalized pressure,
$$\pi^{\lambda}=\frac{(\rl)^{\gamma}-1-\gamma(\rl-1)}{(\gamma-1)} .$$
Moreover we assume the initial data satisfy,
\begin{equation}
\tag {\bf{ID}}
\begin{split}
&\rho^{\lambda}_{t=0}=\rho^{\lambda}_{0}\geq 0, \  \vl|_{t=0}=\Phi_{0}^{\lambda}\notag,\\
&\rl\ul |_{t=0}=m^{\lambda}_{0}, \quad m^{\lambda}_{0}=0 \ \text{on}\ \{x\in \R^{3}\mid \rl_{0}(x)=0\},\notag\\
&\mathcal{E}_{0}=\int_{\R^{3}}\left(\pi^{\la}|_{t=0}+\frac{|\nabla\rho^{\lambda}_{0}|^{2}}{2}+\frac{|m^{\la}_{0}|^{2}}{2\rl_{0}}+\la^{2}|\nabla\Phi^{\la}_{0}|^{2}\right)dx< +\infty.\\
&\int_{\R^{3}}|\nabla\sqrt\rl_{0}|^{2}dx<+\infty.
\end{split}
\label{initial}
\end{equation}

\begin{remark}
\label{r2}
By  \eqref{initial} we get $m^{\la}_{0}$ is bounded in  $H^{-1}(\R^{3})$. In fact  we can rewrite $m^{\la}_{0}$ in the following way 
$$m^{\la}_{0}=\frac{m^{\la}_{0}}{\sqrt{\rl_{0}}}\sqrt{\rl_{0}}\chi_{|\rl_{0}-1|\leq 1/2}+\frac{m^{\la}_{0}}{\sqrt{\rl_{0}}}\frac{\sqrt{\rl_{0}}}{\sqrt{|\rl_{0}-1|}}\sqrt{|\rl_{0}-1|}\chi_{|\rl_{0}-1|> 1/2}$$
then, $m^{\la}_{0}$ is bounded in $L^{2}(\R^{3})+L^{2k/(k+1)}(\R^{3})$, where $k=\min \{2,\gamma\}$ and hence in $H^{-1}(\R^{3})$. 
\end{remark}

The existence of global weak solutions ``\`a la Leray'' for fixed $\la>0$ for the system \eqref{3.1} deserves some comments. One of the main difficulty in the proof of existence of weak solutions for the Navier Stokes equations is the strong compactness of the density in some $L^{p}$ space in order to pass to the limit in the pressure term. In the case of a Navier Stokes Korteweg fluid for the density are available bounds in $L^{\infty}((0,T);\dot H^{1}(\R^{3}))$, hence one can handle easily the convergence of the pressure term. However one is unable to pass to the limit in the quadratic terms of the type $\nabla\rl\otimes\nabla\rl$ appearing in the capillarity tensor and there is also a loss of information for the gradient of the velocity near the vacuum.  The existence of strong solutions for the Navier Stokes Korteweg equations has been obtained by Hattori and Li in \cite{HL96a},  \cite{HL96b}, while the existence of weak solutions has been obtained by Bresch, Desjardins, Lin in \cite{BDL03} where they use some special test functions depending on $\rho$  in order to deal with the vacuum problem. Their result can be easily adapted in order to prove the existence of weak solutions for the system \eqref{3.1}. We summarize this existence result for the system \eqref{3.1} in the following theorem, see \cite{BDD05}.
\begin{theorem}
\label{t1}
Assume \eqref{initial},  and let $\gamma>3/2$, then there exists a global weak solution $(\rl, \ul, \vl)$ to \eqref{3.1} such that $\rl-1\in L^{\infty}((0,T);L^{\gamma}_{2}(\R^{3}))$, $\rl\in L^{2}((0,T);\dot{H}^{2}(\R^{3}))$, $\nabla\rl$ and $\nabla\sqrt{\rl} \in  L^{\infty}((0,T);L^{2}(\R^{3}))$, $\sqrt{\rl}\ul\in L^{2}((0,T);L^{2}(\R^{3}))$, $\sqrt{\rl}D(\ul)\in L^{2}((0,T);L^{2}(\R^{3}))$. 
Furthermore
\begin{itemize}
\item The following energy inequality holds for almost every $t\geq0$,
\begin{equation}
\label{2.5}
\mathcal{E}(t)+\int_{0}^{t}\int_{\R^{3}}\left(\mu|\sqrt{\rl}D\ul|^{2}\right)dxds\leq \mathcal{E}_{0}.
\end{equation}
where
$$\mathcal{E}(t)=\int_{\R^{3}}\left(\rl\frac{|\ul|^{2}}{2}+\pi^{\la}+\frac{|\nabla\rl|^{2}}{2}+\la^{2}|\nabla \vl|^{2}\right)dx.$$
\item The continuity equation $\eqref{3.1}_{1}$ is satisfied in the sense of distribution.
\item For all $\varphi\in\mathcal{D}((0,T)\times\R^{3})$, one has
\begin{align*}
\int_{\R^{3}}&m_{0}\rl_{0}\varphi\!+\!\int_{0}^{T}\!\!\!\!\int_{\R^{3}}\Big((\rl)^{2}\ul\!\cdot\!\partial_{t}\varphi +\rl\ul\otimes\ul\!: \!D(\varphi)-(\rl)^{2}\ul\cdot\varphi\dive \ul\\
&+(\rl)^{\gamma}\dive\varphi-(\rl)^{2}\nabla\vl\cdot\varphi-2\rl D(\ul):\rl D(\ul)\\
&-\rl D(\ul):\varphi\otimes \nabla\rl-(\rl)^{2}\Delta\rl\dive\varphi-2\rl(\varphi\cdot\nabla\rho)\Delta\rl\Big)dxdt=0,
\end{align*}
where $``:''$ denotes the product between matrices.
\end{itemize}
 \end{theorem}
Having collected all the preliminary material we are now ready to state our main result.
\begin{Main}
Let $(\rl, \ul, \vl)$ be a sequence of weak solutions in $\R^{3}$ of the system (\ref{3.1}), assume that the initial data satisfy \eqref{initial}. Then
\begin{itemize}
\item [\bf{(i)}] $\rl\longrightarrow 1$ \quad weakly in $L^{\infty}([0,T];L^{k}_{2}(\R^{3}))$ and strongly in\\ $L^{2/s}((0,T);H^{1+s}_{loc}(\R^{3}))\cap C((0,T);H^{s}_{loc}(\R^{3}))$, $0<s<1$.
  \item [\bf{(ii)}] The gradient component ${\bf H}^{\bot}(\rl\ul)$ of the momentum satisfies
  \begin{equation*}
{\bf H}^{\bot}(\rl\ul)\longrightarrow 0\quad \text{strongly in $L^{q}([0,T];L^{p}(\R^{3}))$}, 
\end{equation*}
where $\displaystyle{q=\frac{4(s_{0}+3)}{2s_{0}+5}}$, $\displaystyle{p=\frac{12(s_{0}+3)}{8s_{0}+19}}$ for any $s_{0}\geq 3/2$.
 \item [\bf{(iii)}] The divergence free component ${\bf H}(\rl\ul)$ of the momentum satisfies
   \begin{equation*}
{\bf H}(\rl\ul)\longrightarrow {\bf H}u=u\quad \text{strongly  in $L^{2}([0,T];L^{p}_{loc}(\R^{3}))$, $1\leq p\leq 3/2$}.
 \end{equation*}
 \item [\bf{(iv)}] $\rl\ul \longrightarrow u$  a.e.
 \item [\bf{(v)}] $u={\bf H}u$ satisfies the following equation
\begin{align}
{\bf H}\Big(\partial_{t} u&-\Delta u+(u\cdot\nabla)u\Big)=0,
\label{nsd}
\end{align}
in $\mathcal{D}'([0,T]\times \R^{3})$.
\end{itemize}
\label{tm1}
\end{Main}
The remaining part of this paper is devoted to the proof of the Main Theorem \ref{tm1}.

\section{Notations and Mathematical tools}
For convenience of the reader we establish some notations and recall some basic facts that will be used in the sequel.
\subsection{Notations}
Given real valued functions   $F,G$,  we write $F\lesssim G$ if  there exists $c\in \R$ such that $F\leq c\ G$. 

We denote by  
\begin{itemize}
\item[a)] $\mathcal{D}(\R ^d \times \R_+)$ the space of test function $C^{\infty}_{0}(\R^d \times \R_+)$, by $\mathcal{D}'(\R^d \times\R_+)$ the space of Schwartz distributions and $\langle \cdot, \cdot \rangle$ the duality bracket between $\mathcal{D}'$ and $\mathcal{D}$.
\item[b)] $W^{k,p}(\R^{d})=(I-\Delta)^{-\frac{k}{2}}L^{p}(\R^{d})$ and $H^{k}(\R^{d})=W^{k,2}(\R^{d})$ the nonhomogeneous Sobolev spaces, for any $1\leq p\leq \infty$ and $k\in \R$. $\dot W^{k,p}(\R^{d})=(-\Delta)^{-\frac{k}{2}}L^{p}(\R^{d})$ and $\dot H^{k}(\R^{d})=\dot W^{k,2}(\R^{d})$  denote the homogeneous Sobolev spaces. The notations $L^{p}_{t}L^{q}_{x}$, $L^{p}_{t}W^{k,q}_{x}$, $C_{t}W^{k,q}_{x}$ will abbreviate respectively  the spaces $L^{p}([0,T];L^{q}(\R^{d}))$, $L^{p}([0,T];W^{k,q}(\R^{d}))$ and  $C([0,T];W^{k,q}(\R^{d}))$.
\item[c)] $L^{p}_{2}(\R^{d})$  the Orlicz space defined as follows
\begin{equation}
\label{2.1}
\hspace{-0,2cm}L^{p}_{2}(\R^{d})\!=\!\{f\!\in\! L^{1}_{loc}(\R^{d})\!\mid \!|f|\chi_{|f|\leq \frac{1}{2}}\in L^{2}(\R^{d}),\ |f|\chi_{|f|> \frac{1}{2}}\in L^{p}(\R^{d})\!\},
\end{equation}
see \cite{Ada75} for more details.
\item[d)] $\bf H$ and $\bf{H}^{\bot}$  the Helmotz Leray projectors, $\bf{H}^{\bot}$ on the space of gradients vector fields and $\bf H$ on the space of divergence - free vector fields. Namely
\begin{equation}
{\bf{H}}^{\bot}=\nabla \Delta^{-1}\dive, \qquad \bf H=I-\bf{H}^{\bot}.
\label{pr}
\end{equation} 
It is well known that   $\bf{H}^{\bot}$ and $\bf H$  can be expressed in terms of Riesz multipliers, therefore they are  bounded linear operators on every $W^{k,p}_{x}$ $(1< p<\infty)$ space (see \cite{Ste93}).   \\ \\
\end{itemize}

\subsection{Mathematical tools}
\subsubsection{Compactness theorems}

In the paper we use also the following  compactness lemmas. The first one is the so called Lions-Aubin Lemma,  (see \cite{A63},  \cite{Si}).

\begin{theorem}
Let be $X, B, Y$ Banach spaces such that $X$ is included in $B$ with compact imbedding and $B\subset Y$ and let be $u_{n}$ a bounded sequence in $L^{p}([0,T]; X)$,  such that $\partial u_{n}/\partial t$ are bounded in   $L^{p}([0,T];Y)$ for $1\leq p<\infty$. Then, $u_{n}$ is relatively compact in  $L^{p}([0,T];B)$.
\label{LA}
\end{theorem}

We will also use the  following generalization of the  div-curl lemma (see Lemma 1.1. in \cite{FLT00})

\begin{lemma}
Assume that $\{u_{n}(\cdot, t)\}$ and $\{v_{n}(\cdot, t)\}$ are vector fields in $\R^{d}$ for $0\leq t\leq T$ such that
\begin{itemize}
\item[L1)] $u_{n}\longrightarrow u$ and $v_{n}\longrightarrow u$ weak-$\ast$ in $L^{\infty}([0,T];L^{2}_{loc}(\R^{d}))$ and strongly in $C([0,T];H^{-1}_{loc}(\R^{d}))$;
\item[L2)] $\{\dive u_{n}\}$ is precompact in $C([0,T];H^{-1}_{loc}(\R^{d}))$;
\item[L3)] $\{\mathrm{curl}\ v_{n}\}$ is precompact in $C([0,T];H^{-1}_{loc}(\R^{d}))$.
\end{itemize}
Then,
$$u_{n}\cdot v_{n}\longrightarrow u\cdot v\qquad \text{in  $\mathcal{D}'([0,T]\times \R^{d})$}.$$
\label{divcurl}
\end{lemma}

\subsubsection{Strichartz estimates for dispersive  equations}
As explained in the Introduction in the sequel we need dispersive estimates for equations describing the acoustic waves. So it is worthwhile to recall the basic facts for this equations.
Following Keel and Tao \cite{KT98}, we start with a very general abstract setting.

Let $(X,dx)$ be measure space and $H$ a Hilbert space and for all $t\in \R$, $U(t):H\rightarrow L^{2}(X)$ ($(U(t))^{\ast}$ it's adjoint) an operator that fulfills the following inequality
\begin{equation}
\label{kt1}
\|U(t)f\|_{L^{2}_{x}}\lesssim \|f\|_{H}
\end{equation}
and for some $\delta>0$ one of the following ``dispersive'' estimate. 

For all $t\neq s$ and all $g\in L^{1}(X)$
\begin{equation}
\label{kt2}
\|U(s)(U(t))^{\ast}g\|_{L^{\infty}}\lesssim |t-s|^{-\delta}\|g\|_{L^{1}}.
\end{equation}
For all $t, s$ and all $g\in L^{1}(X)$
\begin{equation}
\label{kt3}
\|U(s)(U(t))^{\ast}g\|_{L^{\infty}}\lesssim (1+|t-s|)^{-\delta}\|g\|_{L^{1}}.
\end{equation}
\begin{definition}
We say that the exponent pair $(q,r)$ is $\delta$-admissible if $q,r\leq 2$, $(q,r,\delta)\neq (2,\infty,1)$ and
\begin{equation}
\label{kt4}
\frac{1}{q}+\frac{\delta}{r}\leq \frac{\delta}{2}.
\end{equation}
If  equality \eqref{kt4} hold we say that $(q,r)$ is sharp $\delta$- admissible.
\end{definition}
We have then the following Strichartz type estimate (see \cite{KT98}).
\begin{theorem}
\label{KT}
If $U(t)$ obeys \eqref{kt2} and \eqref{kt3}, then the estimates
\begin{equation}
\label{kt5}
\|U(t)f\|_{L^{q}_{x}L^{r}_{t}}\lesssim \|f\|_{H}
\end{equation}
\begin{equation}
\label{kt6}
\|\int (U(s))^{\ast}F(s)ds\|_{H}\lesssim \|F\|_{L^{q'}_{x}L^{r'}_{t}}
\end{equation}
\begin{equation}
\label{kt7}
\|\int U(t)(U(s))^{\ast}F(s)ds\|_{L^{q}_{x}L^{r}_{t}}\lesssim \|F\|_{L^{\tilde{q}'}_{x}L^{\tilde{r}'}_{t}}
\end{equation}
hold for all sharp $\delta$-admissible pairs $(q,r)$, $(\tilde{q},\tilde{r})$.
Furthermore if the decay hypothesis is strengthened to \eqref{kt3}, then \eqref{kt5}, \eqref{kt6}, \eqref{kt7} hold for all $\delta$-admissible pairs $(q,r)$, $(\tilde{q},\tilde{r})$.
\end{theorem}
In the next to section we will apply the Theorem \ref{KT} to the dispersive equations used in the paper.
\subsubsection{Strichartz estimate for Klein Gordon equation}
We apply the previous Theorem \ref{KT} to the following Klein Gordon equation,
\begin{equation*}
\left(-\partial_{tt}+\Delta-m^{2}\right)w(t,x)=F(t,x)\\
\end{equation*}
with Cauchy data 
\begin{equation*}
w(0,\cdot)=f,\quad \partial_{t}w(0,\cdot)=g,
\end{equation*}
where $m>0$ is the mass and $0<T<\infty$.  It turns out that the Klein Gordon operator satisfies the decay estimate \eqref{kt2}-\eqref{kt3} with exponent $\displaystyle{\delta=\frac{d-1}{2}}$, so by applying the Theorem \ref{KT} with $d=3$, $X=\R^{3}$  and $H=L^{2}(\R^{3})$ we get that $w$ satisfies the following Strichartz estimate, (see also  Corollary 2, page 712 in \cite{S77})
\begin{equation*}
\begin{split}
\|w\|_{L^{q}_{t,x}}+\|\partial_{t}w\|_{L^{q}_{t}W^{-1,q}_{x}}+\|w\|_{C_{t}\dot H^{1/2}_{x}}&+\|\partial_{t}w\|_{C_{t}\dot H^{-1/2}_{x}}\\&\lesssim \|f\|_{\dot H^{1/2}_{x}}+\|g\|_{\dot H^{-1/2}_{x}}+\|F\|_{L^{p}_{t,x}},
\end{split}
\label{s2'}
\end{equation*}
where $(q,p)$,  are \emph{admissible pairs}, namely they satisfy 
\begin{equation*}
\frac{4}{3}\leq p \leq \frac{10}{7} \qquad \frac{10}{3}\leq q \leq 4. 
\end{equation*}
By choosing $p=4/3$ and $q=4$ and by a standard application of Duhamel's principle 
it is straightforward to observe that for any $s\in\R$  the following Strichartz  estimate holds,
\begin{equation}
\begin{split}
\|w\|_{L^{4}_{t}W^{s,4}_{x}}+\|\partial_{t}w\|_{L^{4}_{t}W^{-1+s,4}_{x}}&+\|w\|_{C_{t}\dot H^{1/2+s}_{x}}+\|w_{t}\|_{C_{t}\dot H^{-1/2+s}_{x}}\\
&\lesssim \|f\|_{ \dot H^{1/2+s}_{x}}+\|g\|_{ \dot H^{-1/2+s}_{x}}+\|F\|_{L^{1}_{t}\dot H^{s}_{x}}.
\end{split}
\label{s2}
\end{equation}

\subsubsection{Strichartz estimates for the Beam  equations}
The second dispersive equation we use is the so called  Beam  equation,
\begin{equation*}
\left(-\partial_{tt}+\Delta^{2}-m^{2}\right)w(t,x)=F(t,x)\\
\end{equation*}
with Cauchy data 
\begin{equation*}
w(0,\cdot)=f,\quad \partial_{t}w(0,\cdot)=g,
\end{equation*}
The Beam operator verifies the decay estimates \eqref{kt2}-\eqref{kt3} with and exponent $\displaystyle{\delta=\frac{d}{4}}$, see for example \cite{GPW08}, \cite{L98}. Then by applying the Theorem \ref{KT} with $d=3$, $X=\R^{3}$  and $H=L^{2}(\R^{3})$ we get that $w$ satisfies the following Strichartz estimate (for more details see Theorems 3.1 and 3.2 in \cite{L98})
\begin{equation*}
\|w\|_{L^{q}_{t,x}}+\|w\|_{C_{t}L^{2}_{x}}+\|\partial_{t}w\|_{C_{t}\dot H^{-2}_{x}}\leq \|f\|_{\dot H^{2}}+\|g\|_{L^{2}_{x}}+\|F\|_{L^{q}_{t,x}}.
\end{equation*}
where $q$ is \emph{admissible}, namely it satisfy
\begin{equation*}
q\geq 2+\frac{8}{3}
\end{equation*}
As before, by a standard application of Duhamel's principle it is straightforward to see that for any $s\in \R$ the following estimate is also true
\begin{equation}
\|w\|_{W^{s,q}_{t,x}}+\|w\|_{C_{t}\dot H^{s}_{x}}+\|\partial_{t}w\|_{C_{t}\dot H^{s-2}_{x}}\leq \|w_{0}\|_{\dot H^{2+s}}+\|w_{1}\|_{\dot H^{s}_{x}}+\|F\|_{L^{1}_{t}\dot H^{s}_{x}}.
\label{s3}
\end{equation}

\section{Uniform estimates}
\label{S-estimates}
In this section we wish to establish all the basic a priori estimates, independent on $\la$, for the solutions of the system \eqref{3.1}. We point out that all the estimates of this section can be recovered by using smooth approximating solutions constructed by means of  a regularizing process. We skip all the details since this procedure is standard in the literature.
First of all we remind that from the Theorem \ref{t1} we have that the solutions of \eqref{3.1} satisfy the following uniform energy estimate
\begin{equation}
\label{energy}
\mathcal{E}(t)+\int_{0}^{t}\int_{\R^{3}}\left(\mu|\sqrt{\rl}D\ul|^{2}\right)dxds\leq \mathcal{E}_{0}
\end{equation}
where
$$\mathcal{E}(t)=\int_{\R^{3}}\left(\rl\frac{|\ul|^{2}}{2}+\pi^{\la}+\frac{|\nabla\rl|^{2}}{2}+\la^{2}|\nabla \vl|^{2}\right)dx.$$

Beside the standard estimate \eqref{energy} it is possible to recover some further bounds on the second derivative of $\rl$ and on $\nabla\sqrt{\rl}$. First, we need to prove the following lemma.
\begin{lemma}
Assume that $(\rl, \ul, \vl)$ is a global weak solution of \eqref{3.1} and that \eqref{initial} hold, then the  following identity holds,
\begin{equation}
\begin{split}
\frac{1}{2}\frac{d}{dt}\int_{\R^{3}}\rl |\nabla\log\rl|^{2} dx&+\int_{\R^{3}}\nabla\dive \ul\cdot \nabla\rl dx\\
&+\int_{\R^{3}}\rl D(\ul):\nabla\log\rl\otimes\nabla\log\rl dx=0.
\end{split}
\label{ineq1}
\end{equation}
\end{lemma}
\begin{proof}
By dividing first the continuity equation $\eqref{3.1}_{1}$ by $\rl$, then by differentiating it with respect to space we get
\begin{equation}
\partial_{t}(\partial_{i} \log\rl)+ \partial_{i}  \dive\ul+\partial_{i} (\nabla\log\rl\cdot u)=0.
\label{con1}
\end{equation}
The identity \eqref{ineq1} follows now, easily,   by multiplying  \eqref{con1} by $\rl\partial_{i}\log\rl$ and integrating by parts.
\end{proof}
Using the identity \eqref{ineq1} we are able to prove the following uniform estimate,
\begin{proposition}
Assume that $(\rl, \ul, \vl)$ is a global weak solution of \eqref{3.1} and that \eqref{initial} hold, then the solutions of the system \eqref{3.1} satisfy the following inequality,
\begin{equation}
\begin{split}
\frac{1}{2}\frac{d}{dt}\int_{\R^{3}}&\Big(\rl |\ul+\nabla\log\rl|^{2}+|\nabla\rl|^{2}+\lambda|\nabla\vl|^{2}\Big)dx\\
&+4\int_{\R^{3}}\Big(p'(\rl)|\nabla\sqrt{\rl}|^{2}+|\nabla\nabla\rl |^{2}+\frac{(\rl-1)^{2}}{\lambda^{2}}\Big)dx\leq \mathcal{E}_{0}.
\end{split}
\label{ineq2}
\end{equation}
\end{proposition}
\begin{proof}
By multiplying the momentum equation $\eqref{3.1}_{2}$ by $\nabla\rl/\rl$ and by integrating by parts we get
\begin{equation}
\begin{split}
\int_{\R^{3}}4p'(\rl)|\nabla\sqrt{\rl}|^{2}dx&+\int_{\R^{3}}|\nabla\nabla\rl |^{2}dx\\
+\int_{\R^{3}}\partial_{t}\ul\nabla\rl dx+\int_{\R^{3}}\ul&\nabla\ul\nabla\rl dx
+\int_{\R^{3}}\nabla\ul :\nabla\nabla\rl dx\\
-\int_{\R^{3}}\rl D(\ul)&:\nabla\log\rl\otimes\nabla\log\rl dx=\int_{\R^{3}}\nabla\vl\nabla\rl dx.
\end{split}
\label{in1}
\end{equation}
Now by using the Poisson equation $\eqref{3.1}_{3}$we can rewrite the integral in the right hand side of \eqref{in1} as follows,
\begin{equation}
\begin{split}
\int_{\R^{3}}\nabla\vl\nabla\rl dx&=\int_{\R^{3}}\nabla\vl\nabla(\rl-1)dx\\
&=-\int_{\R^{3}}\Delta\vl(\rl-1)dx=-\int_{\R^{3}}\frac{(\rl-1)^{2}}{\lambda^{2}}dx
\end{split}
\label{in2}
\end{equation}
Now by combing together \eqref{in1} with \eqref{in2}  the identity  \eqref{ineq1}  and the energy estimate \eqref{energy} we end up with \eqref{ineq2}.
\end{proof}

\subsection{Consequences of the uniform estimate}
We collect here  all the a priori bounds provided by the energy inequality  \eqref{energy} and from  the uniform estimate \eqref{ineq2}. From \eqref{energy} we get that  there exists $c>0$  depending only from $\mathcal{E}_{0}$, such that
\begin{equation}
\label{c1}
\|\sqrt{\rl}\ul\|_{L^{\infty}_{t}L^{2}_{x}}\leq c,
\end{equation}
\begin{equation}
\label{c2}
\|\sqrt{\rl}D(\ul)\|_{L^{2}_{t}L^{2}_{x}}\leq c,
\end{equation}
\begin{equation}
\label{c3}
\|\nabla\rl\|_{L^{\infty}_{t}L^{2}_{x}}\leq c,
\end{equation}
\begin{equation}
\label{c4}
\|\lambda\nabla\vl\|_{L^{\infty}_{t}L^{2}_{x}}\leq c.
\end{equation}
Since $\pi^{\la}\in L^{\infty}([0,T];L^{1}(\R^{3}))$ it is straightforward to deduce 
\begin{equation}
\label{c5}
\rl-1\quad\text{is bounded in $L^{\infty}([0,T];L^{k}_{2}(\R^{3}))$, where $k=\min(\gamma,2)$},
\end{equation}
\begin{equation}
\label{c5bis}
\rl\quad\text{is bounded in $C([0,T];L^{p}_{loc}(\R^{3}))$, where  $1\leq p< \gamma$},
\end{equation}
Moreover, the additional estimate \eqref{ineq2} provides more regularity on $\rl$, in fact we have that
\begin{equation}
\|\rl\|_{L^{2}_{t}\dot H_{x}^{2}}\leq c.
\label{c6}
\end{equation}
The uniform $L^{\infty}([0,T];L^{2}(\R^{3}))$ bound on $\sqrt{\rl}\nabla\log\rl$ yields to
\begin{equation}
\label{c7}
\|\nabla\sqrt{\rl}\|_{L^{\infty}_{t}L^{2}_{x}}\leq c.
\end{equation}
Finally, from \eqref{ineq2}  it follows  
\begin{equation}
\|\rl-1\|_{L^{2}_{t}L^{2}_{x}}\leq c\lambda.
\label{c8}
\end{equation}

\section{Strong convergence of $\rl$ and $\sqrt{\rl}$}
Here by using  the bounds obtained in Section \ref{S-estimates} we are able to prove some results concerning the  convergence of $\rl$ and $\sqrt{\rl}$.
A straightforward consequence of \eqref{c8} is that
\begin{equation}
\rl -1 \longrightarrow 0\qquad \text {strongly in $L^{2}_{t}L^{2}_{x}$.}
\label{d1}
\end{equation}
\begin{equation}
\sqrt{\rl} -1 \longrightarrow 0\qquad \text {strongly in $L^{2}_{t}L^{2}_{x}$.}
\label{d1bis}
\end{equation}
Now by using together  \eqref{c3}, \eqref{c6} and the fact that 
$$\partial_{t}\rl=-\dive(\rl\ul)\qquad\text{is bounded in $L^{2}_{t}H^{-1}_{x}$},$$
we can apply the Lions-Aubin Lemma \ref{LA} to conclude that for any   compact set $K\subset \R^{3}$,
\begin{equation}
\begin{split}
&\rl\longrightarrow 1\qquad \text{strongly in}\\ 
&\hspace{2cm}\text{$L^{2/s}(0,T;H^{1+s}(K))\cap C(0,T;H^{s}(K))$, $0<s<1$,}
\end{split}
\label{d2}
\end{equation}
moreover, by taking $s>1/2$, the Rellich-Kondrachov Theorem implies
\begin{equation}
\rl\longrightarrow 1\qquad \text{uniformly in $[0,T]\times K$.}
\label{d2bis}
\end{equation}
Finally,  by  combining together \eqref{c3}, \eqref{c6} and \eqref{d2}  we get that
\begin{equation}
\nabla\rl\longrightarrow 0\qquad \text {strongly in $L^{2}_{t}L^{2}_{x}$.}
\label{d33}
\end{equation}
Now, by using \eqref{d1} and  \eqref{d33} we have that
\begin{equation}
\nabla\sqrt{\rl}\longrightarrow 0\qquad \text {strongly in $L^{2}(0,T;L^{2}_{loc}(\R^{3}))$.}
\label{d44}
\end{equation}


\section{Convergence and Compactness of $\rl\ul$}
The next step in our limit analysis is to get enough information  in  order to pass into the limit  in the convective term $\dive(\rl\ul\otimes\ul)$. Unfortunately the estimates of the Section \ref{S-estimates} are not enough to handle this nonlinear term. In fact from the bound \eqref{c1} we get only the weak convergence of $\sqrt{\rl}\ul$ and a new difficulty takes place concerning the loss of information on the gradient of $\ul$ (see the estimate \eqref{c2}) when vacuum appears. So it becomes involved to pass to the limit in the term $\rl\ul\otimes\ul$.  In order to deal with this  loss of information the goal of this section is to get some compactness for the momentum term $\rl\ul$. First of all, by combing together \eqref{c1}, \eqref{c5bis} and \eqref{c7} we have that
\begin{equation}
\rl\ul \quad \text {is uniformly bounded in $L^{\infty}([0,T];L^{p}_ {loc}(\R^{3})$, \quad $1\leq p\leq 3/2$.}
\label{m1}
\end{equation}
On the other hand from 
 \eqref{c1}, \eqref{c2}, \eqref{c7}, \label{d2bis} we get 
\begin{equation}
\nabla(\rl\ul) \qquad \text {is uniformly bounded in $L^{2}([0,T];L^{1}_{loc}(\R^{3})\cap L^{3/2}_{loc}(\R^{3}))$,}
\label{m2}
\end{equation}
hence, by using simoultaneously \eqref{m1} and \eqref{m2} we have
\begin{equation}
\rl\ul\qquad \text {is uniformly bounded in $L^{2}([0,T];W^{1,1}_{loc}(\R^{3})\cap W^{1,3/2}_{loc}(\R^{3}))$.}
\label{m22}
\end{equation}
The previous bound entails only the weak convergence of $\rl\ul$. In order to study the strong convergence we decompose the momentum term in its soleinoidal and gradient part, namely
$$\rl\ul={\bf H}(\rl\ul)+{\bf H}^{\bot}(\rl\ul)$$
and we analyze separately the convergence of these two terms.

\subsection{Compactness of the soleinoidal part ${\bf H}(\rl\ul)$} 
In order to get the compactness of ${\bf H}(\rl\ul)$, by using the Lions-Aubin Lemma \ref{LA} and \eqref{m22},  we need to show that $\partial_{t}{\bf H}(\rl\ul)$ is bounded in $L^{2}_{t}W^{-k,p}_{x}$ for some $k>0$ and $p\geq 1$. From the uniform bounds of Section \ref{S-estimates} we have 
\begin{equation}
\dive(\sqrt{\rl}\ul\otimes \sqrt{\rl}\ul),\  \nabla p(\rl) \in L^{\infty}_{t}W^{-1,1}_{x},
\label{m3}
\end{equation}
\begin{equation}
(\rl-1)\nabla\vl=\dive(\lambda\nabla\vl\otimes\lambda\nabla\vl)+\frac{\lambda^{2}}{2}\nabla|\nabla\vl|^{2}\in L^{\infty}_{t}W^{-1,1}_{x},
\label{m4}
\end{equation}
\begin{equation}
(\rl-1)\nabla\Delta\rl \in  L^{\infty}_{t}W^{-1,2}_{x}.   
\label{m5}
\end{equation}

By applying the Helmotz-Leray projector ${\bf H}$ to the momentum equation we are able to conclude 
\begin{equation}
\partial_{t}{\bf H}(\rl\ul) \quad\text{is bounded in $L^{2}([0,T];W^{-2,4/3}(\R^{3})$.}
\label{m6}
\end{equation}
Finally, \eqref{m6} and the Lemma \ref{LA} yields
\begin{equation}
{\bf H}(\rl\ul) \quad \text{is compact  in $L^{2}([0,T];L^{p}_{loc}(\R^{3}))$,  $1\leq p\leq 3/2$.}
\label{m7}
\end{equation}

Next, since $\sqrt{\rl}\ul$ is uniformly bounded in $L^{2}_{t}L^{2}_{x}$ we deduce that it converges weakly to some $\overline{m}\in L^{2}_{t}L^{2}_{x}$. This fact together with \eqref{d2} allows us to define a limit velocity $u$ as follows
\begin{equation}
\ul(t,x)=\frac{\sqrt{\rl}\ul}{\sqrt{\rl}}\rightharpoonup \overline{m}=u(t,x)\quad \text{in $ L^{2}_{t}L^{2}_{x}$ }.
\label{vel}
\end{equation}
Hence by passing into the limit inside the conservation of mass equation $\eqref{3.1}_{1}$ we get 
\begin{equation}
\dive u=0\qquad \text{in $\mathcal{D}'((0,T)\times \R^{3})$.}
\label{diverg}
\end{equation}
By using together  \eqref{d2}, \eqref{m7}, \eqref{diverg}, it follows
\begin{equation}
{\bf H}(\rl\ul)\longrightarrow {\bf H}u=u \quad \text{strongly  in $L^{2}([0,T];L^{p}_{loc}(\R^{3}))$,  $1\leq p\leq 3/2$.}
\label{m77}
\end{equation}

\subsection{Convergence of  ${\bf H}^{\bot}(\rl\ul)$}

Let us  define the density fluctuation in the usual way
\begin{equation}
\sls=\frac{\rl-1}{\lambda}
\label{fluct}
\end{equation}
then,  by the  identity
\begin{equation}
{\bf H}^{\bot}(\rl\ul)=-\lambda\nabla\Delta^{-1}\partial_{t}\sls,
\label{m8}
\end{equation}
we can deduce that the convergence of ${\bf H}^{\bot}(\rl\ul)$ is strictly related to the one of the density fluctuation. 
As  mentioned in the introduction the weak convergence of the the gradient part of $\rl\ul$ is induced by the so called acoustic waves. In fact as we will see in this section the density fluctuation exhibits very fast oscillating waves in time (the so called plasma oscillation). In order to control this high frequency waves  we will recover the acoustic equation satisfied by $\sls$, we  show that it enjoys various dispersive properties which will enable us to estimate the density fluctuation $\sls$ uniformly with respect to $\la$. 

Let us rewrite the system \eqref{3.1} in the following way
\begin{align}
\partial_{t}\sls+\frac{1}{\lambda}\dive(\rl \ul)&=0\label{4.2.1}\\
\partial_{t}(\rl\ul)+\frac{1}{\lambda}\nabla\sls&=\dive(\rl D(\ul))-\dive(\rl \ul\otimes \ul)-\nabla p(\rl)\notag\\
&+\frac{1}{\lambda}\nabla\sls+(\rl-1)\nabla \vl+\nabla \vl+\rl\nabla\Delta\rl,
\label{4.2.2}\\
\la\Delta \vl&=\sls.
\label{4.2.2bis}
\end{align}
Then, by differentiating with respect to time the equation \eqref{4.2.1}, by taking the divergence of \eqref{4.2.2} and by using \eqref{4.2.2bis} we get that $\sls$ satisfies the following equation
\begin{align}
\lambda^{2}\partial_{tt}{\sls}-\Delta \sls&+\sls=-\lambda\dive\dive\left(\rl D(\ul)-\rl \ul\otimes \ul\right)\label{4.2.3}\\
&-\lambda\dive\left(\!-\nabla p(\rl)+\frac{1}{\la}\nabla\sls+(\rl-1)\nabla \vl+\rl\nabla\Delta\rl\!\!\right).\notag
\end{align}
It turns out that \eqref{4.2.3} is a nonhomogeneous Klein Gordon equation with mass $1/\la$, in order to get uniform bounds on the fluctuation $\sls$ 
we have to take into account the combined description of dispersion and high frequency time oscillations provided by the Strichartz estimates \eqref{s2}.  
In order to make  the equation \eqref{4.2.3} more easier to handle, we rescale the time  variable, the density fluctuation, the velocity and the electric potential in the following way 
\begin{align}
\tau&=\frac{t}{\la}, \quad y=x\label{4.2.4}\\ 
\ut(y,\tau)&=\ul(y,\la\tau), \quad  \rt(y,t)=\rl(y,\la\tau)\notag\\
\st(y,\tau)&=\sls (y,\la\tau),
\quad \Vt(y,\tau)=\vl(y,\la\tau).\label{4.2.5}
\end{align}
As a consequence of this scaling the Klein Gordon equation \eqref{4.2.3} becomes, 
\begin{align}
\partial_{\tau\tau}{\st}-\Delta \st+\st&=\tilde F
\label{k2}
\end{align}
where 
\begin{equation}
\begin{split}
\tilde F&=-\lambda\dive\left(\dive(\rt D(\ut))-\dive(\rt \ut\otimes \ut)-\nabla p(\rt)+(\rt-1)\nabla\Vt\right)\\
&-\lambda\dive\left((\rt-1)\nabla\Delta\rt\right)-\lambda\dive\left(\nabla\Delta\rt+\lambda^{-1}\nabla\st\right)\\
&=\tilde F_{1}+\tilde F_{2}+\tilde F_{3}.
\end{split}
\label{k3}
\end{equation}
By using the uniform bounds of the Section 4, the Poisson equations \eqref{4.2.2bis}, for any $s_{0}\geq3/2$ we have
\begin{equation}
\label{k4}
\begin{split}
\tilde F_{1}=-\lambda\dive\Big(\dive(\rt& D(\ut))-\dive(\rt \ut\otimes \ut)-\nabla p(\rt)\\
&+\dive(\lambda\nabla\Vt\otimes\lambda\nabla\Vt)-\frac{1}{2}\nabla|\la\nabla\Vt|^{2}\Big)\in L^{\infty}_{\tau}H^{-s_{0}-2}_{y},
\end{split}
\end{equation}\\
\begin{equation}
\tilde F_{2}=-\lambda^{2}\dive\nabla(\st\Delta\rl)-\lambda\dive(\nabla\rt\Delta\rt)
\in L^{1}_{t}H^{-s_{0}-2}_{x}+L^{2}_{\tau}H^{-s_{0}-1}_{y},
\label{k5}
\end{equation}\\
\begin{equation}
\tilde F_{3}=-\lambda\dive(\nabla\Delta\rt)-\dive(\nabla\st) \in L^{2}_{\tau}H^{-2}_{y}.
\label{k6}
\end{equation}
Then the following estimate on $\sls$ holds.
\begin{theorem}
Let us consider the solutions $(\rl, \ul, \vl)$ of the Cauchy problem for the system \eqref{3.1} with initial data satisfying \eqref{initial}. Then, for any $s_{0}\geq 3/2$,  the following estimate holds
\begin{align}
&\la^{-\frac{1}{4}}\|\sls\|_{L^{4}_{t} W^{-s_{0}-2,4}_{x}}
+\la^{\frac{3}{4}}\|\partial_{t}\sls\|_{L^{4}_{t} W^{-s_{0}-3,4}_{x}}\notag\\
&+\|\sls\|_{C_{t}H^{-3/2-s_{0}}_{x}}+\lambda\|\partial_{t}\sls\|_{C_{t}H^{-5/2-s_{0}}_{x}}\notag\\
&\lesssim \|\sls_{0}\|_{H^{-1}_{x}}+\|m^{\la}_{0}\|_{H^{-1}_{x}}\notag\\
&+T\|\dive(\dive(\rl D(\ul))-\dive( \rl\ul\otimes\ul)-\nabla p(\rl)+(\rl-1)\nabla\vl)\|_{L^{\infty}_{t}H^{-s_{0}-2}_{x}}\notag\\
&+\lambda\|\dive\nabla(\sls\Delta\rl)\|_{L^{1}_{t}H^{-s_{0}-2}_{x}}+\sqrt{T}\|\dive(\nabla(\rl-1)\Delta\rl)\|_{L^{2}_{t}H^{-s_{0}-1}_{x}}\notag\\
&+\sqrt{T}\|\dive(\nabla\Delta\rl+\lambda^{-1}\nabla(\rl-1))\|_{L^{2}_{t}H^{-2}_{x}}.
\label{k7}
\end{align}
\end{theorem}

\begin{proof}
By using the bounds \eqref{k4}-\eqref{k6} and in the same spirit of \cite{DM12},  we apply the Strichartz estimate \eqref{s2} with $(y,\tau)\in \R^{3}\times(0,T/\la)$ to the scaled Klein Gordon equation \eqref{k2} and we get 
\begin{align*}
&\|\st\|_{L^{4}_{\tau} W^{-s_{0}-2,4}_{y}}+\|\partial_{\tau}\st\|_{L^{4}_{\tau} W^{-s_{0}-3,4}_{y}}\\
&+\|\st\|_{C_{\tau}H^{-3/2-s_{0}}_{y}}+\|\partial_{t}\st\|_{C_{\tau}H^{-5/2-s_{0}}_{y}}\\&\lesssim 
\|\st_{0}\|_{H^{-3/2-s_{0}}_{y}}+\|\partial_{\tau}\st_{0}\|_{H^{-5/2-s_{0}}_{y}}\notag\\
&+T\|\dive (\dive(\rt\ut\otimes\ut)+\nabla p(\rt)-\dive(\rt D(\ut))-(\rt-1)\nabla\Vt)\|_{L^{\infty}_{\tau}H^{-s_{0}-2}_{y}}\notag\\
&+\lambda^{2}\|\dive\nabla(\st\Delta\rt)\|_{L^{1}_{\tau}H^{-s_{0}-2}_{y}}+\sqrt{\lambda}\sqrt{T}\|\dive(\nabla(\rt-1)\Delta\rt)\|_{L^{2}_{\tau}H^{-s_{0}-1}_{y}}\notag\\
&+\sqrt{\lambda}\sqrt{T}\|\dive(\nabla\Delta\rt+\lambda^{-1}\nabla(\rt-1)\|_{L^2_{\tau}H^{-2}_{y}}.
\label{k8}
\end{align*}
Finally,  since 
\begin{equation*}
\|\st\|_{L^{q}_{\tau}W^{k,p}_{y}}=\la^{-\frac{1}{q}}\|\st\|_{L^{q}_{t}W^{k,p}_{x}}
\end{equation*}
by  using that $\partial_{t}\sigma_{0}=m_{0}$ together with the Remark \ref{r2} and  $\sigma_{0}=\lambda\Delta\Phi_{0}\in H^{-1}_{x}$ we end up with \eqref{k7}.
\end{proof}
Going back to \eqref{m8} we get 
\begin{equation*}
\begin{split}
\|{\bf H}^{\bot}(\rl\ul)\|_{L^{4}_{t}W^{-s_{0}-2,4}_{x}}&\leq\lambda^{1/4}\|\la^{3/4}\nabla\Delta^{-1}\partial_{t}\sls\|_{L^{4}_{t}W^{-s_{0}-2,4}_{x}}\\
&\leq
\lambda^{1/4}\|\la^{3/4}\partial_{t}\sls\|_{L^{4}_{t}W^{-s_{0}-3,4}_{x}},
\end{split}
\end{equation*}
and by using \eqref{k7} we end up with, 
\begin{equation}
{\bf H}^{\bot}(\rl\ul)\longrightarrow 0 \quad \text{strongly in $L^{4}_{t}W^{-s_{0}-2,4}_{x}$, for any $s_{0}\geq 3/2$.}
\label{m9bis}
\end{equation}
By using \eqref{m2}, \eqref{m8} and \eqref{m9bis}, by standard interpolation (see Theorem 6.4.5, in \cite{BL76}) it follows
\begin{equation}
{\bf H}^{\bot}(\rl\ul)\longrightarrow 0 \quad \text{strongly in $L^{q}([0,T];L^{p}_{loc}(\R^{3}))$}, 
\label{m9}
\end{equation}
where $\displaystyle{q=\frac{4(s_{0}+3)}{2s_{0}+5}}$,  $\displaystyle{p=\frac{12(s_{0}+3)}{8s_{0}+19}}$, for any $s_{0}\geq 3/2$.

\section{Convergence of the electric field}

This section is devoted to the study of the convergence of the electric field $\el=\nabla\vl$. By the a priori estimate \eqref{c4} we  know that  $\la \el$ is bounded in $L^{\infty}_{t}L^{2}_{x}$ which does not give enough information to pass into the limit in the quadratic term $(\rl-1)\nabla\vl\sim\dive(\la\el\otimes \la\el)-1/2\nabla|\la\el |^{2}$, appearing in the righthand side of $\eqref{3.1}_{2}$. Hence, the problem now, is how to recover the weak continuity of this quadratic forms in $L^{2}$. A way  to recover some weak continuity for scalar product of $L^{2}$ sequences is given by a  compensated compactness tool  as the div-curl lemma. For this purpose we have to recover compactness in space and time.
A key observation follows from the Poisson equation $\eqref{3.1}_{3}$ written in terms of electric field and density fluctuation
\begin{equation}
\la\el=\nabla\Delta^{-1}\sls,
\label{e1}
\end{equation}
where,  by using \eqref{c4} and \eqref{c8} we have 
\begin{equation}
\la\el\qquad \text{is bounded in $L^{2}(0,T;H^{1}(\R^{3}))$.}
\label{e2}
\end{equation}
The previous bound gives us compactness in space but not in time. On the other hand the dispersion of the acoustic equation of Klein Gordon type does not gives us sufficient information.
One way to overcome this further difficulty would be to exploit in a  better way the dispersive behavior of all the terms appearing in the momentum equation. In the previous section we focused on the dispersion given by the combination of the electric field and the Poisson equation, now we are going to exploit the dispersive properties induced by  the capillarity term $\rl\nabla\Delta\rl$. This will be done in the next section.

\subsection{Beam equation for the density fluctuation}

We rewrite the system \eqref{3.1} in the following way
\begin{align}
&\partial_{t}\sls+\frac{1}{\lambda}\dive(\rl \ul)=0\label{e3}\\
\partial_{t}(\rl\ul)&=\dive(\rl D(\ul))-\dive(\rl \ul\otimes \ul)-\nabla p(\rl)\notag\\
&+(\rl-1)\nabla \vl+\nabla \vl+(\rl-1)\nabla\Delta\rl+\nabla\Delta\rl,
\label{e4}\\
\la\Delta \vl&=\sls.
\label{e5}
\end{align}
By differentiating \eqref{e3} with respect to time   and by taking the divergence of \eqref{e4} we get 
\begin{equation}
\begin{split}
\partial_{tt}{\sls}+\Delta^{2} \sls+\frac{1}{\lambda^{2}}\sls=&-\frac{1}{\lambda}\dive\left(\rl D(\ul) -\dive(\rl \ul\otimes \ul)-\nabla p(\rl)\right)\\
&-\frac{1}{\lambda}\dive\left((\rl-1)\nabla \vl+(\rl-1)\nabla\Delta\rl\right).
\end{split}
\label{e6}
\end{equation}
The equation \eqref{e6} goes under the name of Beam equation. In order to handle it in a more easier way  we rescale the time and space variables, the density fluctuation, the velocity and the electric potential in the following way 

\begin{align}
\tau&=\frac{t}{\la}, \quad y=\frac{x}{\sqrt{\la}}\label{4.2.4a}\\ 
\utt(y,\tau)&=\ul(\sqrt{\la} y,\la\tau), \quad  \rtt(y,t)=\rl(\sqrt{\la}y,\la\tau)\notag\\
\stt(y,\tau)&=\sls (\sqrt{\la}y,\la\tau),
\quad \Vtt(y,\tau)=\vl(\sqrt{\la}y,\la\tau).\label{4.2.5b}
\end{align}
Then the equation \eqref{e6} becomes

\begin{equation}
\partial_{\tau \tau}{\stt}+\Delta^{2} \stt+\stt=\tilde{\tilde F}
\label{e7}
\end{equation}
where
\begin{equation}
\begin{split}
\tilde{\tilde F}&=-\dive\left(\dive(\rtt D(\utt)) -\dive(\rtt \utt\otimes \utt)-\nabla p(\rtt)\right)
\\&-\dive\left((\rtt-1)\nabla \Vtt\right)-\frac{1}{\lambda}\dive\left((\rtt-1)\nabla\Delta\rtt\right)\\
&={\tilde{\tilde F}}_{1}+{\tilde{\tilde F}}_{2}+{\tilde{\tilde F}}_{3}.
\end{split}
\label{e8}
\end{equation}

By taking into account the scaling \eqref{4.2.4a} and \eqref{4.2.5b} and the uniform bounds of Section \ref{S-estimates}  for any $s_{0}\geq 3/2$ we have 

\begin{equation}
{\tilde{\tilde F}}_{1}=-\dive(\dive(\rtt D(\utt))-\dive(\rtt \utt\otimes \utt)-\nabla p(\rtt)) \in L^{\infty}_{\tau}H^{-s_{0}-2}_{y},
\label{e9}
\end{equation}\\
\begin{equation}
{\tilde{\tilde F}}_{2}=\dive(\dive(\sqrt{\lambda}\nabla\Vtt\otimes\sqrt{\lambda}\nabla\Vtt)+1/2\nabla|\sqrt{\la}\nabla\Vtt|^{2})\in L^{\infty}_{\tau}H^{-s_{0}-2}_{y},
\label{e10}
\end{equation}\\
\begin{equation}
\begin{split}
\lambda\tilde F_{3}=-\dive\nabla((\rtt-1)\Delta\rtt)-\dive(\nabla\rtt\Delta\rtt)\in L^{1}_{\tau}H^{-s_{0}-2}_{y}+L^{2}_{\tau}H^{-s_{0}-2}_{y}
\end{split}
\label{e11}
\end{equation}

By using the Strichartz estimates for the Beam equation we are able to prove the following theorem

\begin{theorem}
Let us consider the solutions $(\rl, \ul, V^{\la})$ of the Cauchy problem for the system \eqref{3.1} with initial data satisfying \eqref{initial}. Then for any $s_{0}\geq 3/2$,  the following estimate holds
\begin{align}
&\la^{-\frac{1}{4}-\frac{5}{2q}}\|\sls\|_{L^{q}_{t} \dot W^{-s_{0}-2,q}_{x}}
+\|\sls\|_{C(0,T;\dot H^{-s_{0}-2}_{x})}+\|\partial_{t}\sls\|_{C(0,T;\dot H^{-s_{0}-4}_{x})}\notag\\
&\lesssim \la \|\sls_{0}\|_{H^{-1}_{x}}+\|m^{\la}_{0}\|_{H^{-1}_{x}}\notag\\
&+T\|\dive(\dive( \rl\ul\otimes\ul)-\nabla p(\rl)+\dive(\rl D\ul))\|_{L^{\infty}_{t}\dot H^{-s_{0}-2}_{x}}\notag\\
&+T\|\dive (\dive(\la\nabla\vl\otimes\nabla\vl)+\frac{1}{2}\nabla|\la\nabla\vl|^{2})\|_{L^{\infty}_{t}\dot H^{-s_{0}-2}_{x}}\notag\\
&+\|\la^{-1}\dive\nabla((\rl-1)\Delta\rl)\|_{L^{1}_{t}\dot H^{-s_{0}-2}_{x}}+\sqrt{T}\sqrt{\la}\|\dive\nabla\rt\Delta\rt\|_{L^{2}_{t}\dot H^{-2}_{x}}.
\label{e12}
\end{align}
\end{theorem}
\begin{proof}
By using the bounds \eqref{e9}-\eqref{e11}, as in the previous section, we  apply the Strichartz estimate \eqref{s3} with $(y,\tau)\in \R^{3}\times(0,T/\la)$ to the scaled Beam equation \eqref{e7} and we get that $\stt$ verifies
\begin{align*}
&\|\stt\|_{L^{q}_{\tau} W^{-s_{0}-2,q}_{y}}+\|\stt\|_{C(0,T;H^{-s_{0}-2}_{y})}+\|\partial_{t}\stt\|_{C(0,T;H^{-s_{0}-4}_{y})}\notag\\&\lesssim 
\|\st_{0}\|_{H^{-s_{0}}_{y}}+\|\partial_{\tau}\st_{0}\|_{H^{-s_{0}-2}_{y}}\notag\\
&+\frac{T}{\lambda}\|\dive (\dive(\rtt\utt\otimes\utt)-\nabla p(\rtt)+\rtt D(\utt))\|_{L^{\infty}_{\tau}H^{-s_{0}-2}_{y}}\notag\\
&+\frac{T}{\lambda}\|\dive(\dive(\sqrt{\la}\nabla\Vtt\otimes\sqrt{\lambda}\nabla\Vtt)+\frac{1}{2}\nabla|\sqrt{\lambda}\nabla\Vtt|^{2})\|_{L^{\infty}_{\tau}H^{-s_{0}-2}_{y}}\notag\\
&+\frac{1}{\sqrt{\lambda}}\|\dive\nabla\left(\frac{(\rt-1)}{\lambda^{1/4}}\frac{\Delta\rtt}{\lambda^{1/4}}\right)\|_{L^{1}_{\tau}H^{-s_{0}-2}_{y}}+\frac{\sqrt{T}}{\sqrt{\lambda}\lambda^{1/4}}\|\dive\left(\frac{\nabla\rtt}{\sqrt{\la}}\frac{\Delta\rtt}{\lambda^{1/4}}\right)\|_{L^{2}_{\tau}H^{-s_{0}-2}_{y}}\notag\\
\label{e13}
\end{align*}
Finally,  since 
\begin{equation*}
\|\st\|_{L^{q}_{\tau}W^{k,p}_{y}}=\la^{-\frac{1}{q}+\frac{k}{2}-\frac{3}{2p}}\|\se\|_{L^{q}_{t}W^{k,p}_{x}}
\end{equation*}
and, as before, by using the Remark \ref{r2} and the fact that $\sigma_{0}=\lambda\Delta\Phi_{0}\in H^{-1}_{x}$ we end up with \eqref{e12}.
\end{proof}

With the uniform estimate \eqref{e12} we are able to prove the following compactness results concerning $\lambda\nabla\vl$.
\begin{proposition}
Let $(\rl,\ul, \vl)$ be  a sequence of solutions of the Navier Stokes Korteweg Poisson system \eqref{3.1} which satisfy \eqref{initial}, then it holds
\begin{equation}
\lambda\nabla\vl\otimes\lambda\nabla\vl \rightharpoonup 0 \qquad \text{in $\mathcal{D}'([0,T]\times \R^{3})$}.
\label{e14}
\end{equation} 
\label{pelectric}
\end{proposition}
\begin{proof}
In order to prove \eqref{e14} we apply the generalized Div-curl Lemma \ref{divcurl} to the sequences $u_{n}=v_{n}=\lambda\nabla\vl=\la\el$. So we have to check that the hypotheses $L1-L3$ hold.
By combing together  \eqref{c1}, \eqref{e1} and \eqref{e12} we get that
\begin{equation}
\lambda\el\longrightarrow 0  \qquad \text{weak-$\ast$ in $L^{\infty}([0,T];L^{2}_{loc}(\R^{3}))$.}
\label{e15}
\end{equation}
Then, we observe that the hypothesis $L3$ is automatically fullfilled since $curl(\la \el)=0$. In order to verify the hypothesis $L2$ we see that by the Poisson equation
$$\dive(\la\el)=\lambda\nabla\vl=\sls.$$
By using \eqref{e12} we have $\partial_{t}\sls\in C(0,T;H^{-s}(\R^{3}))$, for any $s>1$, so $\sls$ is bounded in $Lip(0,T;H^{-s}(\R^{3}))$, $s>1$ which together with the energy bounds on $\sls$  in $L^{2}(0,T;L^{2}(\R^{3}))$, yields to  the precompactness of $\dive(\la \el)$ in 
$C([0,T];H^{-1}_{loc}(\R^{3}))$. In a similar way we fulfill the  hypothesis $L1$, by combing   $\lambda\el=\Delta^{-1/2}\sls$ and the bounds \eqref{e12}. Since $L1$-$L3$ holds we can conclude by using the Lemma \ref{divcurl} that
$$\lambda\nabla\vl\otimes\lambda\nabla\vl \rightharpoonup 0 \qquad \text{in $\mathcal{D}'([0,T]\times \R^{3})$}.$$
\end{proof}

\section{Proof of the Main Theorem \ref{tm1}}
\begin{itemize}
\item[{\bf(i)}] It follows from \eqref{c5} and \eqref{d2}.
\item[{\bf(ii)}] It follows from \eqref{m9}.
\item[{\bf(iii)}] It follows from \eqref{m77}.
\item[{\bf(iv)}] It follows from (ii) and (iii).
\item[{\bf(v)}] First of all we apply the Leray projector $\bf H$ to the momentum equation of the system \eqref{3.1}, then we have
\begin{equation}
\begin{split}
\partial_{t}{\bf H}(\rho^{\lambda}& u^{\lambda})+{\bf H}(\dive(\rho^{\lambda} u^{\lambda}\otimes u^{\lambda}))\\
&={\bf H}(\dive(\rho D(\ul)+\dive(\la\el\otimes\la\el)+(\rl-1)\nabla\Delta\rl).
\end{split}
\end{equation}
By using together \eqref{m77} and the Proposition \ref{pelectric} for any $\varphi\in \mathcal{D}([0,T]\times \R^{3})$ we obtain that
\begin{equation}
\langle \partial_{t}{\bf H}(\rl\ul)-{\bf H}\dive(\la\el\otimes\la\el),\varphi\rangle\longrightarrow \langle \partial_{t}{\bf H}u,\varphi\rangle.
\label{l1} 
\end{equation}
The convergence established in  \eqref{d2}  entails that for any $\varphi\in \mathcal{D}([0,T]\times \R^{3})$ 
\begin{equation}
\begin{split}
\langle{\bf H} ((\rl-1)\nabla\Delta\rl),\varphi\rangle=&-\langle\nabla (\rl-1)\Delta\rl,{\bf H}\varphi\rangle \\
&-\langle(\rl-1)\Delta\rl, \nabla {\bf H}\varphi\rangle \longrightarrow 0.
\end{split}
\label{l2}
\end{equation}
The convergence of the diffusive terms follows in the following way.
\begin{equation}
\begin{split}
\langle{\bf H}\dive(\rl D(\ul)),\varphi\rangle=&\langle\rl\ul,D(\nabla{\bf H}\varphi)\rangle+\langle\nabla\rl\cdot\ul, \nabla{\bf H}\varphi\rangle\\
&=\langle\rl\ul,D(\nabla{\bf H}\varphi)\rangle+2\langle\sqrt{\rl}\ul\nabla\sqrt{\rl}, \nabla{\bf H}\varphi\rangle\\
&\longrightarrow \langle {\bf H}(\Delta\ul),\varphi\rangle,
\label{l3}
\end{split}
\end{equation}
where we used \eqref{d44}, \eqref{m77} and \eqref{m9}.
For  the convergence of  the convective term is enough to notice that by (i) and (iii) we have that $\sqrt{\rl}$ and $\rl\ul$ converges almost everywhere hence 
\begin{equation}
\sqrt{\rl} \ul=\frac{\rl\ul}{\sqrt{\rl}} \longrightarrow u \quad \text{almost everywhere}.
\end{equation}
And, as a consequence
\begin{align}
\langle {\bf H}\dive(\rho^{\lambda} u^{\lambda}\otimes u^{\lambda}), \varphi\rangle\longrightarrow \langle {\bf H}\dive( u\otimes u),\varphi\rangle\label{l4}
\end{align}

So, by using together \eqref{l1}, \eqref{l2}, \eqref{l3}, \eqref{l4}  we have that $u={\bf H}u$ satisfies the following equation  in $\mathcal{D}'([0,T]\times \R^{3})$
$$
{\bf H}\Big(\partial_{t} u-\Delta u+(u\cdot\nabla)u\Big)=0.
$$
\end{itemize}

 
\bibliographystyle{amsplain}

\begin{thebibliography}{10}

\bibitem{Ada75}
R.~A. Adams, \emph{Sobolev spaces}, Academic Press, New York, 1975.

\bibitem{A63}
J.-P.~Aubin, \emph{Un th\'eor\`eme de compacit\'e}, C. R. Acad. Sci. Paris, \textbf{256}, (1963), 5042--5044.
   
\bibitem{BL76}
J.~Bergh and J.~L\"ofstr\"om, \emph{Interpolation Spaces}, Springer-Verlag, Berlin, Heidelberg, New York, 1976.


\bibitem{BDL03}
D.~Bresch, B.~Desjardins, C.K. ~Lin, \emph{On some compressible fluid models: Korteweg, lubrication and shallow water systems}, Comm. Partial Differential Equations, \textbf{28},  (2003) no. 3--4,  1009--1037.

\bibitem{BDD05}
D.~Bresch, B. Desjardins and B. Ducomet, \emph{Quasi-neutral limit for a viscous capillary model of plasma}, Ann. Inst. Henri, Anal. Nonlinear, \textbf{22}, (2005), 1099--1113.


\bibitem{CDM13} 
L.~Chen, D.~Donatelli and P.~Marcati. 
\emph{Incompressible type limit analysis of a hydrodynamic model for charge-carrier transport.}  SIAM J. Math. Anal. {\bf 45}, (2013), no. 3, 915 - 933.


\bibitem{CDMS96}
S.~Cordier, P.~Degond, P.~Markowich, and C.~Schmeiser, \emph{Travelling wave
  analysis of an isothermal {E}uler-{P}oisson model}, Ann. Fac. Sci. Toulouse
  Math. (6) \textbf{5} (1996), no.~4, 599--643.

\bibitem{CG00}
S.~Cordier and E.~Grenier, \emph{Quasineutral limit of an {E}uler-{P}oisson
  system arising from plasma physics}, Comm. Partial Differential Equations
  \textbf{25} (2000), no.~5-6, 1099--1113.


\bibitem{DFN10}
D.~Donatelli, E.~Feireisl, and A.~Novotn{\'y}
\emph{On incompressible limits for the {N}avier-{S}tokes system on unbounded domains
  under slip boundary conditions}, Discrete Contin. Dyn. Syst. Ser. B
  \textbf{13} (2010), no.~4, 783--798. 


\bibitem{DoFeNo12}
D.~Donatelli, E.~Feireisl, and A.~Novotn{\'y}.
\emph{On the vanishing electron-mass limit in plasma hydrodynamics in
  unbounded media}, J. Nonlinear Sci., \textbf {22}, (2012), no.~6, 985--1012.

\bibitem{DM08}
D.~Donatelli and P.~Marcati, \emph{A quasineutral type limit for the
  {N}avier-{S}tokes-{P}oisson system with large data}, Nonlinearity \textbf{21}
  (2008), no.~1, 135--148.

\bibitem{DM12}
D.~Donatelli and P.~Marcati, \emph{Analysis of oscillations and defect measures for the quasineutral limit in plasma physics}, Arch. of Rat. Mech. and Analysis \textbf{206}
  (2012), no.~1, 159--188.
  
  
  
\bibitem{DFN12}
D.~Donatelli, E.~Feireisl, and A.~Novotn{\'y}, On the vanishing
  electron-mass limit in plasma hydrodynamics in unbounded media, {\it J Nonlinear
  Sci }, (2012).
  

\bibitem{DS85}
J.-E.~Dunn, J.~Serrin, \emph{On the thermomechanics of interstitial working}, Arch. Rational Mech. Anal. \textbf{88} (1985) no. (2), 95--133. 

\bibitem{FLT00}
M.~C.~Lopes Filho, H.~J.~Nussenzveig Lopes and E.~Tadmor,
 \emph{Approximate solutions of the incompressible {E}uler equations with no concentrations}, Ann. Inst. H. Poincar\'e Anal. Non Lin\'eaire., \textbf{17} (2000), no. 3,  371--412
   
\bibitem{GM01a}
I.~Gasser and P.~Marcati, \emph{The combined relaxation and vanishing {D}ebye
  length limit in the hydrodynamic model for semiconductors}, Math. Methods
  Appl. Sci. \textbf{24} (2001), no.~2, 81--92.

\bibitem{GM01b}
I.~Gasser and P.~Marcati, \emph{A vanishing {D}ebye length limit in a hydrodynamic model for
  semiconductors}, Hyperbolic problems: theory, numerics, applications, Vol. I,
  II (Magdeburg, 2000), Internat. Ser. Numer. Math., 140, vol. 141,
  Birkh\"auser, Basel, 2001, pp.~409--414.

\bibitem{GM03}
I.~Gasser and P.~Marcati, \emph{A quasi-neutral limit in the hydrodynamic model for charged
  fluids}, Monatsh. Math. \textbf{138} (2003), no.~3, 189--208. 


\bibitem{GV95}
J.~Ginibre and G.~Velo, \emph{Generalized {S}trichartz inequalities for the
  wave equation}, J. Funct. Anal. \textbf{133} (1995), no.~1, 50--68.

\bibitem{GR95}
R.~J. Goldston and P.~H. Rutherford, \emph{Introduction to plasma physics},
  Institute of Physics Publishing, Bristol and Philadelphia, 1995.

\bibitem{GPW08}
Z.~Guo, L.~Peng, and B.~Wang, Baoxiang, \emph{Decay estimates for a class of wave equations}, J. Funct. Anal., \textbf{254},
(2008) no. 6, 1642--1660.

\bibitem{HL96a}
H.~Hattori H, D.~Li, \emph{The existence of global solutions to a fluid dynamic model for materials
for Korteweg type}, J. Partial Differential Equations, \textbf{9}  (1996), no. 4,  323--342.

\bibitem{HL96b}
H.~Hattori H, D.~Li, \emph{Global solutions of a high-dimensional system for Korteweg materials},
J. Math. Anal. Appl., \textbf{198} (1996), no. 1,  84--97.




\bibitem{JW06}
S.~Jiang and S.~Wang, \emph{The convergence of the {N}avier-{S}tokes-{P}oisson
  system to the incompressible {E}uler equations}, Comm. Partial Differential
  Equations \textbf{31} (2006), no.~4-6, 571--591.

\bibitem{JLW08}
Q.~Ju, F.~Li, and S.~Wang, \emph{Convergence of the {N}avier-{S}tokes-{P}oisson
  system to the incompressible {N}avier-{S}tokes equations}, J. Math. Phys.
  \textbf{49} (2008), no.~7, 073515, 8. 

\bibitem{KT98}
M.~Keel and T.~Tao, \emph{Endpoint {S}trichartz estimates}, Amer. J. Math.
  \textbf{120} (1998), no.~5, 955--980.
  
 \bibitem{K1901}
D.-J.~Korteweg, \emph{Sur la forme que prennent les \'equations du mouvement des fluides si l'on tient compte des forces capillaires par des variations de densit\'e}. Arch. N\'eer. Sci. Exactes S\'er. II \textbf{6},  (1901), 1--24. 

\bibitem{L98}
P.~Levandosky, \emph{Decay estimates for fourth order wave equations}, J. Differential Equations, textbf{143},
(1998) no. 2, 360--413.
     
\bibitem{LY14}
Y. P.~ Li, W.A.~Yong
\emph{Quasi-neutral limit in a 3D compressible Navier-Stokes-Poisson-Korteweg model}, IMA J. of Appl. Math. (2014), doi:10.1093.


\bibitem{L-P.L.M98}
P.-L. Lions and N.~Masmoudi, \emph{Incompressible limit for a viscous
  compressible fluid}, J. Math. Pures Appl. (9) \textbf{77} (1998), no.~6,
  585--627.

\bibitem{L05}
G.~Loeper, \emph{Quasi-neutral limit of the {E}uler-{P}oisson and
  {E}uler-{M}onge-{A}mp\`ere systems}, Comm. Partial Differential Equations
  \textbf{30} (2005), no.~7-9, 1141--1167.

\bibitem{M52}
J.C.~Maxwell, \emph{Capillary action}, The Scientific Papers of James Clerk Maxwell, \textbf{2}, 541--597, New York, Dover, 1952.


\bibitem{PWY06}
Y.-J. Peng, Y.-G.Wang, and W.-A. Yong, \emph{Quasi-neutral limit of the
  non-isentropic {E}uler-{P}oisson system}, Proc. Roy. Soc. Edinburgh Sect. A
  \textbf{136} (2006), no.~5, 1013--1026.

\bibitem{Si}
J.~Simon, \emph{Compact sets in the space {$L\sp p(0,T;B)$}}, Ann. Mat. Pura
  Appl. (4) \textbf{146} (1987), 65--96.

\bibitem{Ste93}
E.~M. Stein, \emph{Harmonic analysis: real-variable methods, orthogonality, and
  oscillatory integrals}, Princeton Mathematical Series, vol.~43, Princeton
  University Press, Princeton, NJ, 1993, With the assistance of Timothy S.
  Murphy, Monographs in Harmonic Analysis, III.

\bibitem{S77}
R.~S. Strichartz, \emph{Restrictions of {F}ourier transforms to quadratic
  surfaces and decay of solutions of wave equations}, Duke Math. J. \textbf{44}
  (1977), no.~3, 705--714.
  



%

\bibitem{W04}
S.~Wang, \emph{Quasineutral limit of {E}uler-{P}oisson system with and without
  viscosity}, Comm. Partial Differential Equations \textbf{29} (2004), no.~3-4,
  419--456.
\end{thebibliography}

\end{document}